\documentclass[runningheads]{llncs}
\usepackage[T1]{fontenc}
\usepackage{subcaption}
\usepackage{graphicx}
\usepackage{algorithmic}
\def\0{\bf \0}
\def\1{\bf \1}
\def\A{{\bf A}}

\def\D{{\bf D}}

\def\H{{\bf H}}

\def\0{{\bf 0}}

\def\R{{\bf R}}

\def\T{{\bf T}}
\def\U{{\bf U}}
\def\V{{\bf V}}

\def\a{{\bf a}}
\def\b{{\bf b}}

\def\e{{\bf e}}

\def\h{{\bf h}}

\def\r{{\bf r}}

\def\u{{\bf u}}
\def\v{{\bf v}}

\def\x{{\bf x}}
\def\y{{\bf y}}
\def\z{{\bf z}}

\def\Tr{{\rm T}}
\def\T{{\rm T}}

\newtheorem{algorithm}{Algorithm}
\begin{document}
\title{A Facet Enumeration Algorithm for Convex Polytopes}
%
%
\author{Yaguang Yang\inst{1}} 
%
%
\institute{NASA, Goddard Space Flight Center \\
8800 Greenbelt Rd, Greenbelt, 20771 MD, USA. \\
\email{yaguang.yang@nasa.gov}}
%
\maketitle              
\begin{abstract}
This paper proposes a novel and simple facet enumeration algorithm for convex 
polytopes. The algorithm is implemented in Matlab. Some simple polytopes
with known H-representations and V-representations are used as the test examples. 
The preliminary numerical test shows the effectiveness and efficiency of the proposed
algorithm. Due to the duality between the vertex enumeration problem and facet
enumeration problem, we expect that this method can also be used to solve the
vertex enumeration problem.
\end{abstract}
%
%
%
\section{Introduction}
For convex polytopes, there are two different 
but equivalent representations: (a) H-representation,
and (b) V-representation. H-representation uses a set of 
closed half-spaces to define a convex polytope, i.e., a 
polytope is given as $\A \x \le \b$, where 
$\A \in \R^{m \times d}$ and $\b \in \R^m$ are
given, and $\x \in \R^d$ satisfying the inequality constraints
is the set of points in the 
convex polytope. A $k$-face of a polytope is the set of
$\x$ that meet two conditions: (1) $\A \x \le \b$, and 
(2) for some $d-k$ independent rows of $\A_k$, $\A_k \x = \b$.
For $k=0$, a $0$-face of the polytope is a vertex;
for $k=1$, a $1$-face of the polytope is an edge;
for $k=d-2$, a $(d-2)$-face of the polytope is a ridge;
for $k=d-1$, a $(d-1)$-face of the polytope is a facet.
In the remainder of the paper, we use vertex, edge, and
facet explicitly, but we do not make distinction for other 
$k$-faces in general. V-representation uses a set of
vertices $\mathcal{V}= \{ \v_1, \v_2, \ldots, \v_n \}$ to 
define a convex polytope as a convex hull of $\mathcal{V}$.
Given one of the representations, there is a need, in many
applications \cite{avis,cgaf94,dl97}, to know what the 
other corresponding representation is. The 
transformation from (a) to (b) is known as the vertex 
enumeration and the other from (b) to (a) is known as 
the facet enumeration. 

Vertex enumeration is clearly related to the simplex method
of linear programming. The so-called double-description method
can be dated back to Fourier (1827) \cite{ziegler95} and 
was reinvented by Motzkin \cite{mrtt53}. This method
constructs the polytope sequentially by adding one constraint at
a time. All new vertices produced must lie on the hyperplanes
bounding the constraint currently being inserted. A more
popular vertex enumeration method is based on pivoting, which
was discussed by Chand and Kapur \cite{ck70}, Dyer \cite{dyer83},
Swart \cite{swart85}, among others. The most popular method
in this category is the reverse search method proposed by
Avis and Fukuda \cite{af92}. Their method starts at an
``optimum vertex'' and traces out a tree in depth order
by ``reverse'' Bland's rule \cite{bland77}. 

Noticing that the dual problem of a vertex (resp. facet) enumeration 
problem is the facet (resp. vertex) enumeration problem for 
the same polytope where the input and output are simply 
interchanged, Bremner, Fukuda and Marzetta \cite{bfm98} 
pointed out that the vertex enumeration methods can be 
used for the facet enumeration problem. They 
suggested using a primal-dual approach proposed
in \cite{fr94} that includes two purely combinatorial 
algorithms for enumerating all faces of a d-polytope
based on the combinatorial vertices' description and
some information on edges. Very recently, Avis and Jordan
\cite{aj18} published a scalable parallel vertex/facet 
enumeration code. Avis and Jordan's paper is also a good 
source to find the related works. 

In this paper, we consider the facet enumeration problem
without using duality of the vertex enumeration. To our
best knowledge, this is the first ``direct'' method which
avoids converting back and forth primal and dual problems.
Given the vertices of a convex polytope, 
the idea is to first find all edges of the polytope, which 
provides a vertex/edge diagram that connects any 
vertex to its neighbor vertices. The information is
stored in a matrix $\D$, which can be represented 
as a connected tree. For any vertex, using matrix
$\D$, one can create a branch with depth of $d$. 
If vertices in the branch of the tree from a root vertex 
to an end vertex form a hyperplane, one can 
check if all vertices in $\mathcal{V}$ are in the half
space defined by the hyperplane. If the answer is
true, the hyperplane contains a facet of the polytope,
otherwise, it does not. Repeating these steps for all
vertices will find all half spaces whose intersection
forms the polytope. Using the duality between the 
vertex enumeration problem and facet enumeration
problem, we expect that this method can also be 
used to solve the vertex enumeration problem with some
minor tweaks but that is not the purpose of this paper.

The remainder of the paper is organized as follows:
Section 2 discusses the edge detection. 
Section 3 provides details of the proposed algorithm.
Section 4 presents some numerical examples to 
show the effectiveness and efficiency of the algorithm.
Concluding remarks are summarized in Section 5.
All mathematical proofs are placed in the Appendix.

\section{Edge detection}

Assume that a d-dimensional polytope has $n$ vertices.
Each vertex $\v_i$ is represented as a {\bf row vector}, and
the $n$ vertices are stored in a matrix $\V \in \R^{n \times d}$.  
To have an efficient algorithm,  we set the centroid 
of the polytope $\v_0$ as the origin which is an interior point
of the polytope. Let
\begin{equation}
\u_i = \v_i -\v_0, \hspace{0.1in} i=1, \ldots, n.
\label{shiftP}
\end{equation}
We will denote by $\mathcal{U} = \{ \u_1, \ldots, \u_n \}$
the set of vertices which form a polytope with its center
as the origin of the coordinate system. We also use a matrix 
\[ \U = \left[ \begin{array}{c} \u_1 \\ \vdots \\ \u_n
\end{array} \right] \in \R^{n\times d}
\]
to store these vertices. Our idea is to calculate a {\bf column vector} 
$\h_i \in \R^d$, or the hyperplane $\h_i^{\T} \x =1$, 
where $\x  \in \R^d$ is the set of d-dimensional column 
vectors that spans the hyperplane, so that the $m$ half spaces 
(associated with the $m$ hyperplanes)
\begin{equation}
\H \x := \left[ \h_1, \ldots, \h_m \right]^{\Tr} \x \le \e
\label{Hexpression}
\end{equation}
defines the polytope, where $\e$ is a {\bf column vector} of all 1's. 

The algorithm first establishes a vertex/edge diagram that 
describes all vertex/edge relations of a given polytope. 
The process is as follows: given any vertex $\u_i$
on the polytope,  the proposed algorithm identifies its 
neighbor vertices which are connected by edges. 
Let $\u_i$ and $\u_j$ be two vertices of the polytope.
The line pass both $\u_i$ and $\u_j$ can be defined by
\begin{equation}
\mathcal{L} = \{ \u_i + t (\u_j - \u_i) ~|~ t \in \R \}.
\label{lineEq}
\end{equation}
Assuming that a point $\z$ is on the line of $\mathcal{L}$
and the {\bf row vector} $\z -\0=\z$ (where $\0$ is the origin and
the center of the polytope) is perpendicular to the line 
$\mathcal{L}$, then $\z$ satisfies the following equations:
\begin{equation}
\z = \u_i + t (\u_j - \u_i),  \hspace{0.1in} 
\langle \z, \u_j - \u_i \rangle =0.
\label{perPen}
\end{equation} 
Solving (\ref{perPen}) for $t$ yields
\begin{equation}
t=-\frac{\u_i  (\u_j - \u_i)^{\Tr} } { \| (\u_j - \u_i) \|^2 },
\label{tPara}
\end{equation}
which gives 
\begin{equation}
\z = \u_i - \frac{\u_i (\u_j - \u_i)^{\Tr}  }
{ \| (\u_j - \u_i) \|^2 }(\u_j - \u_i).
\label{zPoint}
\end{equation}

We have the following claim:

\begin{lemma}
$\z  \ne \0$ if $\mathcal{L}$ does not cross the centroid (origin).
\label{zNeo}
\end{lemma}
\begin{proof}
Since $\z$ is on $\mathcal{L}$ and $\mathcal{L}$
does not cross $\0$, it must have $\z \ne \0$.
\hfill \qed
\end{proof}

Thus, $\z=\0$ implies that $\u_j$ is not a neighbor
vertex of $\u_i$ (this will reduce the computational effort
if such a scenario appears). Therefore, we may assume, without
loss of generality, that $\z \ne \0$. Denote a {\bf row vector} 
$\a = \z / \| \z \| \ne \0$ and a constant 
$c=\a \z^{\T} = \frac{\z \z^{\T}}{ \| \z \|} =\| \z \| \ne 0$. 
Then, a hyperplane 
$\mathcal{P}$ passing line $\mathcal{L}$ with the normal 
direction $\a^{\Tr}$ is given by 
\begin{equation}
\a (\x-\z^{\Tr}) =\a   \x -c =0,
\label{pCheck}
\end{equation}
where $\a$ is a known {\bf row vector}, $c$ is a constant, and 
$\x \in \R^d$ is any point on the hyperplane $\mathcal{P}$. 
Dividing both sides by $c$ and denote $\h^{\T} = \a/c$,
i.e., $\h$ is a {\bf column vector}, we get
\begin{equation}
\h^{\T}  \x  = 1.
\label{pPlane}
\end{equation}
Intuitively, from the construction of $\mathcal{P}$, if for
all $\u_k$, we have
\begin{equation}
\h^{\T}  \u_k^{\Tr}   \le 1
\label{uCheck}
\end{equation}
with equality hold for only $\u_i$ and $\u_j$,
then the segment between $\u_i$ and $\u_j$ is an edge of 
the polytope and $\u_i$ and $\u_j$ are adjacent vertices.
We summarize the discussion as the following theorem.

\begin{theorem}
Let $\u_i$ and $\u_j$ be two vertices of a convex polytope.
Let $\mathcal{L}$ be the line as defined in (\ref{lineEq}). 
For $t$ given in (\ref{tPara}), $\z=\u_i + t (\u_j - \u_i)$ 
is perpendicular $\mathcal{L}$. Denote $\a = \z / \| \z \|$ 
and $c=\a \z^{\Tr}$ ($\a$ and $\z$ are {\bf row vectors}). 
Let $\h^{\T} = \a /c$. Then, 
\begin{equation}
\h^{\T} \x  =1, \hspace{0.1in} \forall \x \in \R^d
\label{hyperplane}
\end{equation}
is a hyperplane passing $\u_i$ and $\u_j$, and the
following claims hold:

(a) If inequality (\ref{uCheck}) holds for all $\u_k$, 
and the equality holds for only $\u_i$ and $\u_j$,
then the segment between $\u_i$ and $\u_j$ is an edge of 
the polytope, i.e., $\u_i$ and $\u_j$ are adjacent vertices.

(b) If inequality (\ref{uCheck}) holds for all $\u_k$, 
and the equality holds for more than two but less
than $d$ vertices, then the line segment between 
$\u_i$ and $\u_j$ is on a $k$-face which is
part of the hyperplane.

(c) If inequality (\ref{uCheck}) holds for all $\u_k$, 
and the equality holds for at least $d$ vertices,
then the line segment between $\u_i$ and $\u_j$
is on the hyperplane and a facet of the polytope is
part of the hyperplane.
\label{edgeFinder}
\end{theorem}
\begin{proof}
We show part (a) only because parts (b) and (c) follows
similar argument. Since (\ref{uCheck}) holds for all vertices 
$k=1, \ldots,n$, all points of the polytope are inside of the 
half space $\h^{\T} \x \le 1$. The half space contains the polytope.
Since $\u_i$ and $\u_j$ are on the hyperplane, all points
on the line segment between $\u_i$ and $\u_j$ are on 
the hyperplane. Since $\h^{\T} \u_k^{\T} <0$ holds for all $\u_k$ 
satisfying $\u_i \ne \u_k \ne \u_j$, for any point in the 
convex hull such that 
\[
\sum_{k=1}^{n} \lambda_k \u_k
\]
with at least one $\lambda_k >0$ and $k \notin \{ i, j \}$, we 
have $\h^{\T} \sum_{k=1}^{n} \lambda_k \u_k^{\T} <1$ because 
$\sum_{k=1}^{n} \lambda_k =1$. Therefore, all those
points in the convex polytope are not on the hyperplane.
Since for all points of the polytope, only line segment 
between $\u_i$ and $\u_j$ are on the hyperplane, the line
segment between $\u_i$ and $\u_j$ is an edge of the polytope.
\hfill \qed
\end{proof}

\begin{remark}
For the purpose of facet enumeration, we are only 
interested in cases (a) and (c) because case (a) of 
Theorem \ref{edgeFinder} implies that $\u_j$ is a neighbor
vertex of $\u_i$, which is used to construct vertex/edge
diagram, and case (c) will be recorded to reduce computational
effort.
\end{remark}

If inequality (\ref{uCheck}) does not hold for 
at least one of $\u_k$, $k=1,\ldots,n$, then we
need to be a little bit more careful. In this case,
we have the following theorem.
\begin{theorem}
Let $\u_i$ and $\u_j$ be any two vertices of a convex polytope,
$\y^*$ and $f^*$ be the optimal solution of (\ref{edgeLP}) in appendix. 
Denote that $c=\h^{\T} \left( \frac{\u_j +\u_i}{2} \right)^{\T}$, 
then the following claims hold:

(a) If $f^*<0$, then, the line segment connecting $\u_j$ and
$\u_i$ is an edge of the polytope.

(b) If $f^*=0$  and $\y^*-\r \neq \0$, then, $\u_j$ and $\u_i$ do not form 
an edge of the polytope and the hyperplane contains either a 
$k$-face if less than $d$ vertices are on the hyperplane 
or a facet if at least $d$ vertices are on the hyperplane.

(c) If $f^*>0$, then, $\u_j$ and $\u_i$ does not form 
an edge of the polytope and the hyperplane does not 
form a $k$-face or a facet of the polytope.
\label{thm2}
\end{theorem}
\begin{proof}
Denote the middle point between $\u_j$ and $\u_i$ as a {\bf row vector}
$\r=[r_1,r_2, \ldots,r_d]=\frac{\u_j +\u_i}{2}$, which is also the vector
from the origin to the middle point of $\u_j$ and $\u_i$. 
Therefore, for a {\bf row vector}\footnote{If $ \frac{\u_j +\u_i}{2}=\0$,
then, according to Lemma \ref{zNeo}, $\{\u_i, \u_j\}$
is not an edge and will not be considered further.}
\begin{equation}
\y \neq \r,
\label{yNotEq}
\end{equation}
then, $\y-\r$ is a vector starting from
the middle point of $\u_j$ and $\u_i$ and pointing to $\y$.
We consider all $\y$ such that $\y-\r$
is perpendicular to ${\u_j -\u_i}$. Clearly, $\y$ must
satisfy
\begin{equation}
\Biggl\langle \y-\r,~~\u_j -\u_i \Biggr\rangle
=0,
\label{eq11}
\end{equation}
where $\langle \cdot, \cdot \rangle$ denotes the inner 
product of its arguments\footnote{Here, we abuse the
notation of inner product by not restricting its augments to
be {\bf column vectors}. The only restriction is that they must
have the same dimension.}. This is equivalent to
\begin{equation}
\y (\u_j -\u_i )^{\T}=\r (\u_j -\u_i)^{\T}.
\end{equation} 
We want to find a bounded vector $\y-\r$ such 
that the smallest angle between $\y-\r$ and
$\u_k-\r$, for $k \notin \{ i, j \}$, is maximized, 
which is equivalent to solving the following mini-max problem:
\begin{equation}
\min_{\y} \max_{k\notin \{ i, j \}} \Bigl\langle 
\frac{\y-\r}{\| \y-\r \|}, 
\frac{\u_k-\r}{ \| \u_k -\r\| } \Bigr\rangle.
\label{NProb}
\end{equation} 
However, in stead of considering the nonlinear problem 
(\ref{NProb}), we prefer to deal with a simpler problem.  Let 
\begin{equation}
f=\max_{k\notin \{ i, j \}} \Bigl\langle \frac{\y-\r}{\|  \y-\r \|}, 
\frac{\u_k -\r}{ \| \u_k -r \|} \Bigr\rangle.
\end{equation}
Since $\|  \y-\r \|$ is independent of $k$, 
we may consider the following linear programming problem: 
\begin{eqnarray}
\min_{\y,f} & f & \nonumber \\
s.t. &  (\u_j -\u_i ) \y^{\T}=\r (\u_j -\u_i)^{\T}, &  \nonumber \\
& (\u_k-\r)(\y-\r)^{\T} 
    \leq  f \| \u_k -\r \|,
     & k\notin \{ i, j \}, \nonumber  \\
& -{\bf 1} \le \y-\r \le {\bf 1}, & 
\label{edgeLP}
\end{eqnarray}
where ${\bf 1}$ is a {\bf row vector} of all ones. There are
many efficient algorithms \cite{dantzig63,wright97,yang20} 
for solving the LP problem, which is out of the scope of
our discussion here. Let $\y^*$ and $f^*$
be the solution of (\ref{edgeLP}) ($\y^*$  is a constant
{\bf row vector}), then the {\bf column vector} $\x$ in the following
equation
\begin{equation}
\left( \y^* -\r \right) \left( \x -\r^{\T} \right)=0
\label{newPlane}
\end{equation}
defines a plane that is perpendicular to $\y^* -\r$
and passes the line segment between $\u_i$ and $\u_j$. 
Substituting $\x^{\T}=\u_i$ to the left side of (\ref{newPlane}) 
and applying (\ref{eq11}) gives
\begin{equation}
-\frac{1}{2}
\Biggl\langle \y^* -\r,~~\u_j -\u_i \Biggr\rangle
=0.
\end{equation}
Therefore, $\u_i$ is on the plane defined by (\ref{newPlane}).
Substituting $\x^{\T}=\u_j$ to the left side of
(\ref{newPlane}) and applying (\ref{eq11}) gives
\begin{equation}
\frac{1}{2}
\Biggl\langle \y^* -\r,~~\u_j -\u_i \Biggr\rangle
=0.
\end{equation}
Therefore, $\u_j$ is on the plane defined by (\ref{newPlane}).
Let $\h^{\T}=\y^* -\r$, according to (\ref{yNotEq}), 
it must have $\h \neq \0$. Assuming that the middle point of $\u_i$ and 
$\u_j$ is not in the origin, i.e., $\r \neq \0$,
we can write the plane of (\ref{newPlane}) as  
\begin{equation}
\h^{\T} \x=\h^{\T} \r^{\T} =c \neq 0.
\end{equation}

If the minimum of (\ref{edgeLP}) $f^* < 0$,
it means that all $\u_k$, $k \notin \{ i, j \}$, are not on 
the plane, and they are all on one side of the plane, 
therefore, the  line segment connecting $\u_j$ and
$\u_i$ is an edge of the polytope. 
If $f^* = 0$,
it means that besides $\u_j$ and $\u_i$,
there are additional vertices $\u_k$, $k \notin \{ i, j \}$
on the plane; although all vertices are on the same side
of the plane defined by (\ref{newPlane}), $\u_j$ and 
$\u_i$ do not form an edge of the polytope. 
If $f^* > 0$,
it means that it is impossible to find a plane which 
passes $\u_j$ and $\u_i$, such that all $\u_k$ are on 
the same side of the plane. Therefore, $\u_j$ and $\u_i$ 
does not form an edge of the polytope.
\hfill \qed
\end{proof}

\begin{remark}
If $f^*=0$ and $\y^*-\r =\0$, this solution of (\ref{edgeLP})
does not give us the necessary information to use Theorem \ref{thm2}, 
in this case, we may add an additional inequality constraint:
\begin{equation}
\sum_{k=1}^{d} (y_k - r_k) \geq 1
\label{ymrEq}
\end{equation} 
in the linear programming problem (\ref{edgeLP}) to enforce
$\y^*-\r \neq \0$.
\end{remark}

\begin{remark}
It is worthwhile to point out that the Matlab function {\tt linprog}
oftentimes fails to find the optimal solution of (\ref{edgeLP}), while 
the facet pivot method developed in \cite{yang21} has no 
problem for all tested problems.
\end{remark}

\begin{remark}
Assume that $\y^*\neq \r$ and $f^*=0$ be the solution of (\ref{edgeLP}). Denote the index set 
$\mathcal{K} =\{ k~|~(\u_k-\r)(\y^*-\r)^{\T} =  f^* \| \u_k -\r \|, ~ k\notin \{ i, j \} \}$.
Then, $\mathcal{F}=\{ \u_i, \u_j, \u_{\mathcal{K}} \}$ forms 
a $k$-face. In addition, if $| \mathcal{F} | \geq d$, then, $\mathcal{F}$ forms a facet.
This information will save some computational effort in the next section.
\end{remark}


\section{Facet enumeration algorithm}

Now we discuss the information storage, which is
also important for the algorithm design. Let $\D$ be the 
$n \times n$ adjacency matrix whose $(i,j)$ element is one 
if the line segment between $\u_i$ and $\u_j$ is an edge or 
is zero otherwise. Matrix $\D$ is obtained by the process described
in Theorem \ref{edgeFinder} or Theorem \ref{thm2}. Since $\D$ 
is symmetric, using this property will reduce the computational 
effort to find all edges of the polytope. 

Let us consider a vertex tree associated with $\D$
and starting from vertex $\u_1$, which connects the adjacent 
vertices. Clearly the structure of the tree is defined by the 
matrix $\D$. To efficiently carry out the 
iteration, we denote by $\mathcal{U}_o$ the set of 
vertices for which all relevant facets have been found, 
by $\mathcal{U}_c$ the set of vertices in the current 
iteration for which the associated hyperplanes are 
to be determined, by $\mathcal{U}_t$ the set of
the vertices that are not in $\mathcal{U}_c$ yet and 
therefore whose facets have not been examined.
$\mathcal{U}_o = \emptyset$ and $\mathcal{U}_c = \u_1$
before the iteration starts. The proposed algorithm 
has two loops. The outer loop uses breadth-first search. 
For a vertex $\u_i \in \mathcal{U}_c$, the algorithm in the 
inner loop finds all facets that intersect at $\u_i$, once all facets 
containing $\u_i$ are found, $\u_i$ is moved from $\mathcal{U}_c$ 
to $\mathcal{U}_o$. While finding facets containing 
$\u_i \in \mathcal{U}_c$, new vertices on each of these 
facets become members of $\mathcal{U}_c$. This process 
is repeated for every $\u_i \in \mathcal{U}_c$ until 
for every vertex, its intersecting facets are found, i.e., the 
process will terminate when $\mathcal{U}_o$ includes
all vertices of $\{ \u_1, \ldots, \u_n \}$.

Having the vertex/edge diagram of $\D$ that connects
all the vertices in $\mathcal{V}$, the inner loop uses the 
following method to identify all the facets that intersects
at $\u_i \in \mathcal{U}_c$. Since every facets that contains 
$\u_i$ can be associated with an edge, for each of these edges
directly connected to $\u_i$, we can create a branch of length
$d$ as follows. Note that nonzero $(i,j)$ elements of the 
adjacency matrix $\D$ define the branch under the vertex $\u_i$. 
For each $\u_i$, assuming that $\u_i$ is the $i$th
vertex, we look at the $i$th row of $\D$ to select the next 
vertex $j$ among all $(i,j)=1$ such that the $j$th 
vertex has not been used in the construction of existing 
hyperplanes; once vertex $j$ is selected, we look at
the $j$th row of $\D$ to select the next 
vertex $k$ among all $(j,k)=1$ such that the $k$th 
vertex has not been used in the construction of existing 
facets; repeat this step until $d$ vertices are found. 

Since hyperplane (\ref{hyperplane}) is uniquely defined by 
$\h$, we will loosely use $\h_i \in \R^d$ for the $i$th facet 
of the polytope if the hyperplane $\h_i$ contains a facet. 
Let $\u_{i_1}$ be one of vertices in the vertex set 
$\mathcal{U}_c$ in current iteration. We say that 
$\{ \u_{i_1}, \u_{i_2}, \ldots,\u_{i_d} \}$ is a branch 
of length $d$ in the tree under $\u_{i_1}$ if for 
$i_j \in \{i_2, \ldots, i_{d-1}\}$, $\u_{i_j}$ is connected 
to only $\u_{i_{j-1}}$ and $\u_{i_{j+1}}$ in 
the set of $\{ \u_{i_1}, \u_{i_2}, \ldots,\u_{i_d} \}$.
Given these $d$ vertices $\{\u_{i_1}, \ldots,\u_{i_d} \}$ 
which are on the branch of the tree starting from $\u_{i_1}$, 
if the base matrix formed by $\{\u_{i_1}, \ldots,\u_{i_d} \}$
is linearly independent, one can solve the linear system
\begin{equation}
\left[ \begin{array}{c}
\u_{i_1} \\ \vdots \\ \u_{i_d} 
\end{array} \right] \h_i := \U_i \h_i =\e.
\label{ithPlane}
\end{equation}
for a candidate facet $\h_i$. If $\h_i^{\T} \u_k^{\T}\le 1$
for $k = 1, \ldots, n$, then $\h_i$ is the hyperplane that
includes a facet. Since (\ref{ithPlane}) may also
create an unwanted hyperplane, these unwanted hyperplanes 
can be identified using one of the following rules: first, 
inequality $\h_i^{\T} \u_i^{\T} \le 1$ is violated, i.e.,
\begin{equation}
\U \h_i \le \e
\label{safeGuard}
\end{equation}
does not hold for some $\u_i \in \U$; second, the hyperplane
has been found earlier in this process (in this case,
the hyperplane will not be add to the matrix $\H$). 
If the newly found hyperplane contains a facet of the convex
polytope that is not in the matrix $\H$, it is then added to $\H$.
Otherwise, discard the hyperplane and continue the search in the tree.

For $\u_i \in \mathcal{U}_c$, the idea of the proposed algorithm 
is to examine all branches $\u_{i_1}, \ldots,\u_{i_d}$ 
under $\u_i =\u_{i_1} \in \mathcal{U}_c$ in a
systematically method to reduce the effort to find facets 
associated with $\u_i$ that has not been found. 
However, given the vertex $\u_i \in \mathcal{U}_c$, we may 
not need to calculate all hyperplanes associated with it because
some facets associated with $\u_i$ may be on other 
hyperplanes defined by some $\h_j$
which have been found in current or previous iterations. 

We use cross-polytope of Example 4 to describe the process, 
for which the $\D$ matrix is obtained by using the method of the previous 
section as:
\begin{equation}
\D = \left[ \begin{array}{cccccccc}
0~~&     1  ~~&   1   ~~&  1   ~~&  1  ~~&   1 ~~&    1~~&     0 \\
     1~~&     0 ~~&    1 ~~&    1  ~~&   0  ~~&   1 ~~&    1  ~~&   1 \\
     1 ~~&    1  ~~&   0   ~~&  1 ~~&    1  ~~&   0   ~~&  1  ~~&   1 \\
     1  ~~&   1  ~~&   1 ~~&    0 ~~&    1 ~~&    1  ~~&   0  ~~&   1 \\
     1  ~~&   0  ~~&   1 ~~&    1  ~~&   0  ~~&   1  ~~&   1 ~~&    1 \\
     1  ~~&   1  ~~&   0 ~~&    1 ~~&    1 ~~&    0  ~~&   1  ~~&   1 \\
     1 ~~&    1 ~~&    1 ~~&    0  ~~&   1  ~~&   1 ~~&    0  ~~&   1 \\
     0  ~~&   1  ~~&   1  ~~&   1 ~~&    1 ~~&    1 ~~&    1  ~~&   0 \\
\end{array} \right].
\end{equation}
\begin{figure}[htb]
\centerline{\includegraphics[height=5cm,width=10cm]{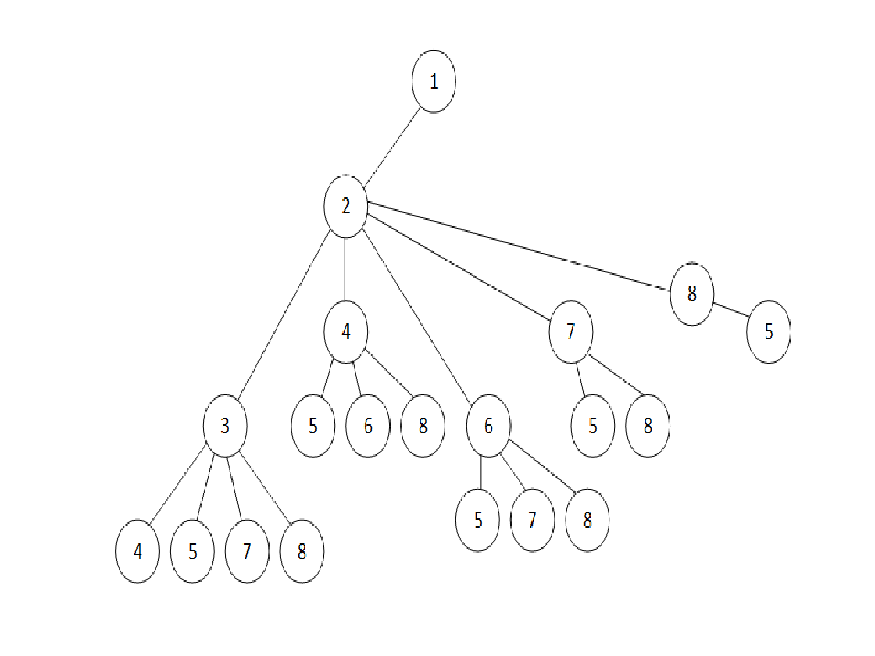}}
\caption{The branches under $\u_i =\u_{1} \in \mathcal{U}_c$.}
\label{treeBranch}
\end{figure}
Assume $\mathcal{U}_c = \{ \u_1 \}$, the branches are constructed as follows: 
In view of the first row of $\D$, $\u_1$ is connected to $\u_2, \u_3, \u_4, \u_5, \u_6, \u_7$. 
We consider the branch with its first edge between $\u_1$ and $\u_2$ as an example. 
The second row of $\D$ indicates
that $\u_2$ is also connected to $\u_3, \u_4, \u_6, \u_7, \u_8$, we select the first
vertex such that the branch includes $\u_1, \u_2, \u_3$. The third row of $\D$
indicates that $\u_3$ is connected to $\u_4, \u_5, \u_7, \u_8$, again, we select
the first vertex such that $\u_1, \u_2, \u_3, \u_4$ forms the first branch with 
length $d=4$ . Using the aforementioned method, we can examine if this 
branch forms a hyperplane. If it does, save it; otherwise, discard it. To construct
the next branch, we move one level up at a time and check if there is any vertex
that has not been considered on this level. In this case, $\u_5$ is the first vertex that 
has not been considered, a new branch $\u_1, \u_2, \u_3, \u_5$ is formed; and 
the new branch is examined to see if it forms a new hyperplane. Repeating
this process, the order of the selected branches is described in 
Figure \ref{treeBranch} that is from the left to the right. 
After all the branches under $\u_1$ has been examined, $8$ facets are found. 
Set $\mathcal{U}_0 = \{ \u_1 \}$. From Figure \ref{treeBranch}, it is clear
that $\mathcal{U}_c = \{ \u_2, \u_3, \u_4, \u_5, \u_6, \u_7, \u_8 \}$. The 
above process can be repeated for every $\u_i \in \mathcal{U}_c$ until
$\mathcal{U}_0 = \{ \u_1, \u_2, \u_3, \u_4, \u_5, \u_6, \u_7, \u_8 \}$.

Denote the number of members of a set $\mathcal{U}$
by $| \mathcal{U} |$. The proposed algorithm is therefore 
given as follows:

\begin{algorithm}  {\ } \\
\begin{algorithmic}[1] 
\STATE Data: Vertex matrix $\V$. 
\STATE Initial step: Calculate centroid $\v_0$ and vertex set 
	$\mathcal{U}$, adjacent matrix $\D$, and initial 
	facet matrix $\H$. Set up initial $\mathcal{U}_0$,
	$\mathcal{U}_c$, and $\mathcal{U}_t$.
\WHILE{$| \mathcal{U}_0 | < n$}
	\STATE set $\mathcal{U}_2 = \emptyset$.
	\FOR{$i=1$: $| \mathcal{U}_c |$}
		\STATE Step 1: Start from $\u_{i_1}$, from row $i_1$ of $\D$, find the first vertex
			$\u_{i_2}$ that has not been considered in the branches under $\u_{i_1}$ so 
			that the new branch includes $\u_{i_1}$ and $\u_{i_2}$.
		\STATE Step 2:  Repeat the process until a new branch $\u_{i_1}, \ldots, \u_{i_d}$
			is formed.
		\STATE Step 3: Solve (\ref{ithPlane}) to get a candidate facet $\h_i$. Test if $\h_i$
			is indeed a new facet. If it is, add it to $\H$; otherwise discard it. 
		\STATE Step 4: Add new vertices to the set $\mathcal{U}_2$ whose members are on the
			hyperplanes of the new branches but are not in $\mathcal{U}_c$.
		\STATE Step 5: Move one level up at a time in the current branch until either (1) there 
			is a vertex on this level that has not been included in a branch under $\u_{i_1}$,
			then repeat Steps 2-4; or (2) all branches under $\u_{i_1}$ have been investigated, 
			increase $i$ by 1.
	\ENDFOR
	\STATE Update $\mathcal{U}_o =\mathcal{U}_o \cup \mathcal{U}_c$. Set
		$\mathcal{U}_c=\mathcal{U}_2$.
\ENDWHILE
\STATE Recover the polytope by shifting the origin $\H (\x  - \v_0) \le \e.$
\end{algorithmic}
\label{searchTree}
\end{algorithm}


\section{Numerical examples}

Several examples are provided in this section.

\begin{example}
The first example is a $2$-dimensional convex polytope
which is a triangle. Its vertices are given by:
\[
\V= \left[ \begin{array} {cc}
0~~ & 0\\ 3~~ &  0 \\ 0~~ &  3
\end{array} \right]
\]
\end{example}
The centroid is $(1,1)$. Algorithm \ref{searchTree}
finds the H-representation as 
\[
\H=  \left[ \begin{array} {rr}
0~~ & -1 \\ -1~~ & 0 \\ 1~~ &  1
\end{array} \right]
\hspace{0.1in}
\b=  \left[ \begin{array} {c}
0\\  0\\  3
\end{array} \right]
\hspace{0.1in}
\]

\begin{example}
The second example is a $3$-dimensional convex polytope 
which is a cubic. Its vertices are given by:
\[
\V= \left[ \begin{array} {rrr}
   0~~ & 0~~ & 0 \\
   0~~ & 0~~ & 1 \\
   0~~ &  1~~ & 0 \\
   1~~ & 0~~ & 0 \\
   0~~ & 1~~ & 1 \\
   1~~ & 0~~ & 1 \\
   1~~ & 1~~ & 0 \\
   1~~ & 1~~ & 1
\end{array} \right]
\]
\end{example}
The centroid is $(0.5,0.5,0.5)$. Algorithm \ref{searchTree}
finds the H-representation as 
\[
\H=  \left[ \begin{array} {rrr}
-2~~   &  0~~   &  0  \\   
     0~~  &  -2~~  &   0  \\  
     0~~   &  0~~  &  -2   \\  
     0~~   &  0~~  &   2   \\  
     0~~   &  2~~  &   0  \\   
     2~~   &  0~~  &   0  \\   
\end{array} \right]
\hspace{0.1in}
\b=  \left[ \begin{array} {c}
0 \\  0 \\  0  \\  2 \\ 2  \\2
\end{array} \right]
\hspace{0.1in}
\]

\begin{example}
The third example is a $3$-dimensional convex polytope 
which is a octahedron. Its vertices are given by:
\[
\V= \left[ \begin{array} {rrr}
   0~~ & 0~~ & 1 \\
   -1~~ & 0~~ & 0 \\
   0~~ &  -1~~ & 0 \\
   1~~ & 0~~ & 0 \\
   0~~ & 1~~ & 0 \\
   0~~ & 0~~ & -1 
\end{array} \right]
\]
\end{example}
The centroid is $(0,0,0)$. Algorithm \ref{searchTree}
finds the H-representation as 
\[
\H=  \left[ \begin{array} {rrr}
-1~~  &  -1~~  &  1     \\
     1~~  &  -1~~   &  1     \\
     1~~   &  1~~   &  1     \\
    -1~~   &  1~~   &  1     \\
    -1~~   & -1~~  &  -1     \\
     1~~  &  -1~~  &  -1     \\
     1~~  &   1~~  &  -1     \\
    -1~~   &  1~~  &  -1       
\end{array} \right]
\hspace{0.1in}
\b=  \left[ \begin{array} {c}
1 \\  1 \\  1  \\ 1 \\ 1  \\ 1 \\ 1  \\ 1
\end{array} \right]
\hspace{0.1in}
\]

\begin{example}
The fourth example is a $4$-dimensional convex polytope 
which is a cross-polytope. Its vertices are given by:
\[
\V= \left[ \begin{array} {rrrr}
     0~~  & 0~~  & 0~~  & 1    \\
    -1~~ &  0~~ &  0~~ &  0    \\
     0~~ & -1~~ &  0~~ &  0    \\
     0~~ &  0~~ & -1~~ &  0    \\
     1~~  & 0~~ &  0~~ &  0    \\
     0~~  & 1~~ &  0~~ &  0    \\
     0~~ &  0~~ &  1~~ &  0    \\
     0~~ &  0~~ &  0~~ &  -1 
\end{array} \right]
\]
\end{example}
The centroid is $(0,0,0,0)$. Algorithm \ref{searchTree}
finds the H-representation as 
\[
\H =  \left[ \begin{array} {rrrr}
    -1~~  &   -1~~  &  -1~~  &   1     \\
     -1~~  &  -1~~  &   1~~  &   1     \\
    - 1~~  &   1~~  &  -1~~  &   1     \\
     -1~~  &  1~~  &   1~~  &   1     \\
     1~~  &   -1~~  &   -1~~  &   1     \\
     1~~  &  -1~~  &   1~~  &   1     \\
     1~~  &   1~~  &  -1~~  &   1     \\
     1~~  &   1~~  &  1~~   &  1     \\
    -1~~  &  -1~~  &   1~~  &  -1     \\
    -1~~  &   1~~  &   1~~  &  -1     \\
     1~~  &  -1~~  &   1~~  &  -1     \\
     1~~  &   1~~  &   1~~  &  -1     \\
    -1~~  &  -1~~  & -1~~  &  -1     \\
    -1~~  &   1~~  & -1~~  &  -1     \\
     1~~  &  -1~~  & -1~~  &  -1     \\
     1~~  &   1~~  &  -1~~  &  -1          
\end{array} \right]
\hspace{0.1in}
\b=  \left[ \begin{array} {c}
1 \\  1 \\  1  \\ 1 \\ 1  \\ 1 \\ 1  \\ 
1 \\  1  \\ 1 \\ 1  \\ 1  \\ 1  \\ 1 \\   1 \\   1
\end{array} \right]
\hspace{0.1in}
\]

For the above problems, using Theorem \ref{edgeFinder}
can find all the edges of the corresponding polytopes.
However, for the following convex polytope, we need 
Theorem \ref{thm2} to find some edges.

\begin{example}
\label{review}
The last example is a $3$-dimensional convex polytope. 
Its vertices are given by:
\[
\V= \left[ \begin{array} {rrrr}
     1~~  & 0~~  & 1    \\
	1~~  & 0~~  & -1    \\
	1.25~~  & -1~~  & 1    \\
	1.25~~  & -1~~  & -1    \\
	0.25~~ &  -1~~ &  1    \\
     0.25~~ &  -1~~ &  -1    \\
    -1~~ &  0~~ &  1    \\
     -1~~ &  0~~ &  -1    \\
     -1.25~~  & 1~~  & 1    \\
	-1.25~~  & 1~~  & -1    \\ 
	-0.25~~ &  1~~ &  1    \\
     -0.25~~ & 1~~ &  -1   
\end{array} \right]
\]
\end{example}
The centroid is $(0,0,0)$. The polytope is depicted
as in Figure \ref{threeD}. Each vertex is numbered 
and its coordinate is provided in the figure. There are 
$12$ vertices, $8$ facets and $18$ edges in this polytope. 
Let $\{i,j\}$ represent the line segment that connects 
vertices $i$ and $j$. Using Theorem \ref{edgeFinder}, we 
can determine $14$ edges. However, for the rest $4$ edges, 
$\{1,2\}$, $\{5,6\}$, $\{7,8\}$, and $\{11,12\}$, 
we have to use Theorem \ref{thm2} to identify them.

\begin{figure}
\centering
\begin{subfigure}{0.48\textwidth}
    \includegraphics[width=\textwidth]{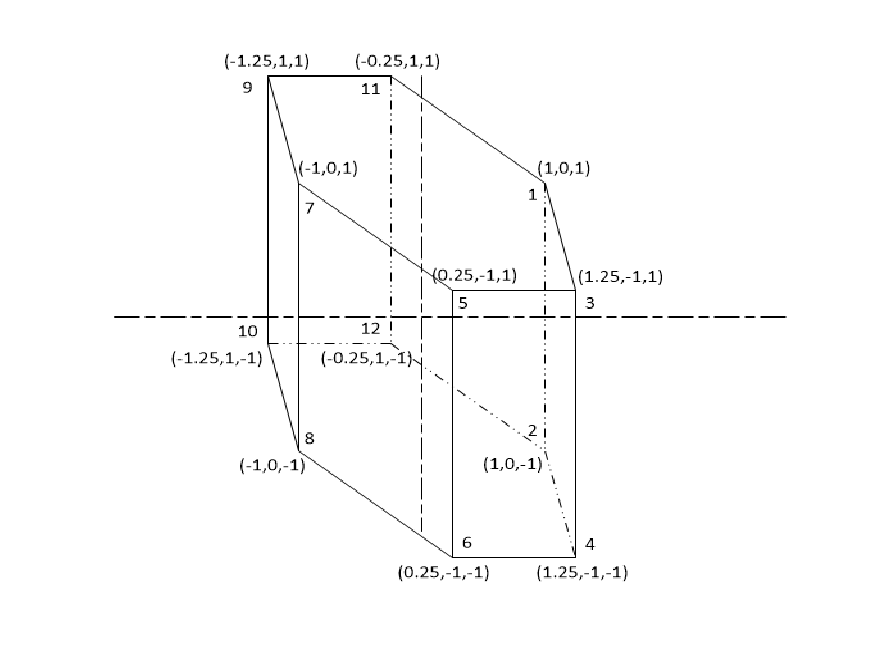}
    \caption{The polytope of Problem \ref{review}.}
    \label{threeD}
\end{subfigure}
\hfill
\begin{subfigure}{0.48\textwidth}
    \includegraphics[width=\textwidth]{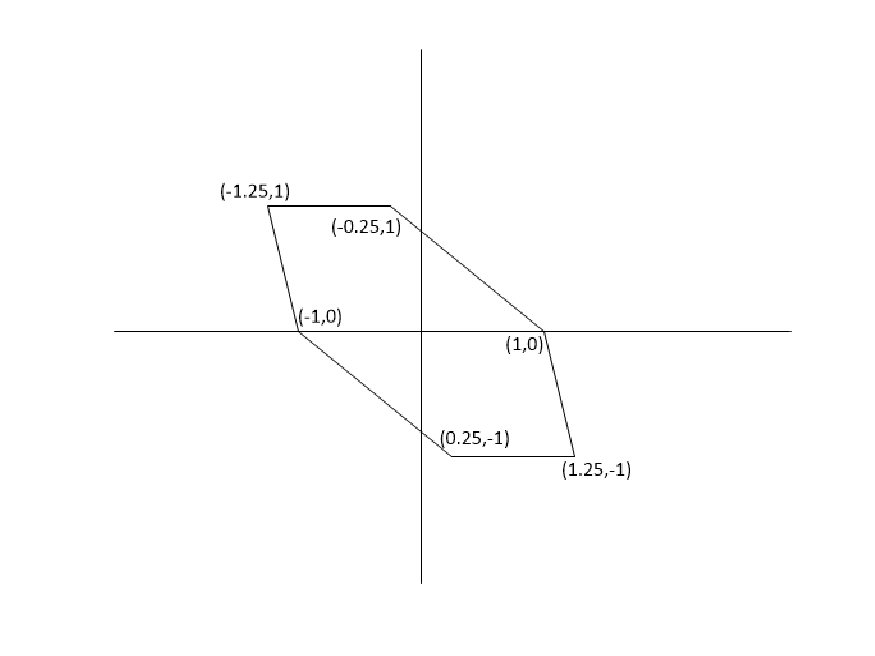}
    \caption{The projected figure of Problem \ref{review}.}
    \label{plane}
\end{subfigure}
\end{figure}

%

The projected figure of the polytope in x-y plane is depicted 
in Figure \ref{plane} which may help us to verify the result 
to be presented. First, let's consider the $\{1,2\}$ edge. 
Given $i=1$ and $j=2$, i.e., $\u_1=[1,0,1]$ and 
$\u_2=[1,0,-1]$, using Theorem 1, we cannot determine if
$\{1,2\}$ is an edge. It is not difficult (but tedious) to verify 
that the corresponding linear programming 
problem (\ref{edgeLP}) can be written as
\begin{equation}
\begin{array}{cl}
\min_{\y,f}  & f  \\
s.t. & \left[ \begin{array}{rrr}
0.25~~ & -1~~ & 1 \\
0.25~~ & -1~~ & -1 \\
-0.75~~ & -1~~ & 1 \\
-0.75~~ & -1~~ & -1 \\
-2~~ & 0~~ & 1 \\
-2~~ & 0~~ & -1 \\
-2.25~~ & 1~~ & 1 \\
-2.25~~ & 1~~ & -1 \\
-1.25~~ & 1~~ & 1 \\
-1.25~~ & 1~~ & -1 
\end{array} \right] 
\left[ \begin{array}{c}
y_1-1 \\ y_2 \\ y_3 
\end{array} \right] 
\leq f \left[ \begin{array}{c}
1.43614 \\ 1.43614 \\ 1.60078 \\ 1.60078  \\ 2.23607 \\ 2.23607 \\ 2.65753 \\ 2.65753  \\ 1.88746 \\ 1.88746
\end{array} \right], \\
   & [0,~0,~-2] [y_1, ~y_2, ~y_3] ^{\T}=0, \\
  & \left[ \begin{array}{c}
0 \\ -1 \\ -1 
\end{array} \right] \le
\left[ \begin{array}{c}
y_1 \\ y_2 \\ y_3 
\end{array} \right] \le
 \left[ \begin{array}{c}
2 \\ 1 \\ 1 
\end{array} \right]
\end{array}
\end{equation}
which has the optimal solution $\y^*=[2,~0.6821,~0]$ and $f^*=-0.300878$. 
According to Theorem \ref{thm2}.(a), the line segment $\{1,2\}$ is an edge.

Given $i=1$ and $j=5$, i.e., $\u_1=[1,0,1]$ and 
$\u_2=[0.25,-1,1]$, using Theorem 1, we cannot determine if
$\{1,5\}$ is an edge. It is not difficult (but tedious) to verify 
that the corresponding linear programming 
problem (\ref{edgeLP}) can be written as
\begin{equation}
\begin{array}{cl}
\min_{\y,f}  & f  \\
s.t. & \left[ \begin{array}{rrr}
0.375~~ & 0.5~~ & -2 \\
0.625~~ & -0.5~~ & 0 \\
0.625~~ & -0.5~~ & -2 \\
-0.375~~ & -0.5~~ & -2 \\
-1.625~~ & 0.5~~ & 0 \\
-1.625~~ & 0.5~~ & -2 \\
-1.875~~ & 1.5~~ & 0 \\
-1.875~~ & 1.5~~ & -2 \\
-0.875~~ & 1.5~~ & 0 \\
-0.875~~ & 1.5~~  & -2 
\end{array} \right] 
\left[ \begin{array}{c}
y_1-0.625 \\ y_2+0.5 \\ y_3 -1
\end{array} \right] 
\leq f \left[ \begin{array}{c}
2.0954 \\ 0.8004 \\ 2.1542 \\ 2.0954  \\ 1.7002 \\ 2.625 \\ 2.4012 \\ 3.125  \\ 1.7366 \\ 2.6487
\end{array} \right], \\
   & [-0.75,~-1,~0] [y_1, ~y_2, ~y_3] ^{\T}=0.0313, \\
  & \left[ \begin{array}{c}
-0.375 \\ -1.5 \\ 0 
\end{array} \right] \le
\left[ \begin{array}{c}
y_1 \\ y_2 \\ y_3 
\end{array} \right] \le
 \left[ \begin{array}{c}
1.625 \\ 0.5 \\ 2 
\end{array} \right]
\end{array}
\end{equation}
which has the optimal solution $\y^*=[2,~0.6821,~0]$ and $f^*=-0.300878$. 
According to Theorem \ref{thm2}.(a), the line segment $\{1,2\}$ is not an edge.

Using the same strategy, we can determine that $\{5,6\}$, $\{7,8\}$, and
$\{11,12\}$ are also edges. Algorithm \ref{searchTree}
finds the H-representation as 
\[
\H =  \left[ \begin{array} {rrrr}
    0~~  &   0~~     &    1     \\
    1~~  & 0.25~~   &    0     \\
   0~~   &    -1~~     &    0     \\
   -1~~  & -1.25~~  &    0     \\
   -1~~  &  -0.25~~  &   0    \\
   0~~   &    1~~     &    0     \\
     1~~  & 1.25~~   &   0    \\
     0~~  &   0~~  &  -1      
\end{array} \right]
\hspace{0.1in}
\b=  \left[ \begin{array} {c}
1 \\  1 \\  1  \\ 1 \\ 1  \\ 1 \\ 1  \\ 1 
\end{array} \right]
\hspace{0.1in}
\]

In summary, for all tested convex polytopes with known H-representations
and V-representation, the algorithm is verified to be successful.

\section{Conclusion}

In this paper, an intuitive and novel facet enumeration
algorithm is proposed. The idea is purely based on
geometric observation, therefore, it is easy to understand.
The outer loop of algorithm is based on breadth-first 
search which eventually covers all the vertices of the
polytope. The inner loop of the algorithm is based on
depth-first search which will find the facets associated
with the vertex under the consideration. The algorithm
is implemented in Matlab and numerical test shows
the efficiency and effectiveness of the algorithm.
Due to the duality between the vertex enumeration problem 
and facet enumeration problem, we expect that this method 
can also be used to solve the vertex enumeration problem 
with some minor tweaks.

\end{document}